\documentclass[11pt]{amsart}
\usepackage{amsmath, amssymb, latexsym, amsthm,bm}
\usepackage{mathtools}
\usepackage{url}
\usepackage[T1]{fontenc}
\usepackage{mathrsfs}
\usepackage{color}
\usepackage{tikz}
\usepackage{graphicx}
\usepackage{amsbsy}
\usepackage{enumitem}
\usepackage{cite}

\def\BG{{\mathbf{\Gamma}}}

\makeatletter \@namedef{subjclassname@2020}{%
\textup{2020} Mathematics Subject Classification} \makeatother

\usepackage{xcolor}
\usepackage[hidelinks]{hyperref}

\usetikzlibrary{calc}
\usepackage{xspace}

\usepackage[mathlines]{lineno}

\usepackage[font=small]{caption}
\usepackage{mathtools,amsmath}
\usepackage{array}

\def\sphere{{{\mathbb S}^2}}

\def\rev#1{{\overleftarrow{#1}}}
\def\Neg#1{{{#1}^{-}}}

\renewcommand{\descriptionlabel}[1]%
{{\hglue -0.7 cm}\hspace{\labelsep}#1}

\newcommand*{\Scale}[2][4]{\scalebox{#1}{\ensuremath{#2}}}%

\newtheorem{theorem}{Theorem} 
\newtheorem{proposition}[theorem]{Proposition} 

\newtheorem{corollary}[theorem]{Corollary}
\newtheorem{lemma}[theorem]{Lemma}

\newtheorem{observation}[theorem]{Observation}

\theoremstyle{definition}

\def\sphere{{{\mathbb S}^2}}

\title[Deciding if a shadow resolves into a given link: linear-time algorithms]{Deciding if a shadow resolves into a given link: \\ linear-time algorithms}

\author[Alba]{Andrea Alba} \address{Instituto de F\'\i sica, Universidad Aut\'onoma de San Luis Potos\'{\i}, SLP 78000, Mexico} \email{\tt andrea.casillas@if.uaslp.mx}

\author[Salazar]{Gelasio Salazar} \address{Instituto de F\'\i sica, Universidad Aut\'onoma de San Luis Potos\'{\i}, SLP 78000, Mexico} \email{\tt gelasio.salazar@uaslp.mx}


\date{August 19, 2026}

\begin{document}

\begin{abstract} A {\em shadow} (or {\em projection}) is obtained from a link diagram by ignoring the over/under information at each crossing. Given a fixed link $L$ we investigate the complexity of deciding whether an input shadow $S$ can be {\em resolved} into $L$, that is, whether we can assign over/under information to its crossings to obtain a diagram of a link isotopic to $L$. We show that if $L\in\{3_1,4_1,5_1,5_2,6_2,L2a1, L4a1, L5a1, L6n1\}$ then there exists a linear-time algorithm that decides whether an input shadow $S$ resolves into $L$. \end{abstract}

\maketitle

\section{Introduction} \label{sec:intro}

We work in the piecewise linear category, and links are hosted in the $3$-sphere ${\mathbb S}^3$. 

We recall that a {\em diagram} $D$ of a link $L$ is a regular projection of $L$ onto the $2$-sphere $\sphere$, together with over/under information at each of its crossings.  

If we ignore the over/under information in $D$ we obtain a {\em shadow} (or {\em projection}) $S$. Since the crossings of $S$ are transverse double points, we may regard $S$ as a $4$-regular graph embedded in $\sphere$. This is quite convenient as it allows us to use standard graph-theoretical terminology.

\vglue 0.4 cm
\noindent{\bf Remark. }{\em Throughout this paper shadows are regarded as $4$-regular graphs embedded in $\sphere$.}
\vglue 0.4 cm

We {\em resolve} a shadow $S$ by assigning over/under information to each vertex of $S$. The resulting diagram is a {\em resolution} of $S$. The shadow $S$ {\em resolves into} a link $L$ if it has a resolution that is a diagram of a link isotopic to $L$. See Figure~\ref{fig:0001} for an illustration.

% **************************************************************
\begin{figure}[htbp] 
\def\ta#1{{\Scale[2.5]{#1}}} 
\def\tb#1{{\Scale[4.5]{#1}}} 
\def\te#1{{\Scale[2.0]{#1}}} 
\def\somea{{\Scale[4.8]{\text{\rm (a)}}}} 
\def\someb{{\Scale[4.8]{\text{\rm (b)}}}} 
\def\somec{{\Scale[4.8]{\text{\rm (c)}}}} 
\def\somed{{\Scale[4.8]{\text{\rm (d)}}}} 
\def\somee{{\Scale[4.8]{\text{\rm (e)}}}} 
\def\somef{{\Scale[4.8]{\text{\rm (f)}}}} 
\def\pluschords{{\Scale[4.4]{${\text{\rm plus all its possible chords}}$}}} 
\def\reds{\scalebox{5}{\small\rmfamily resolve}}
\def\asdi{\scalebox{5}{\small\rmfamily differently}}
\centering 
\scalebox{0.2}{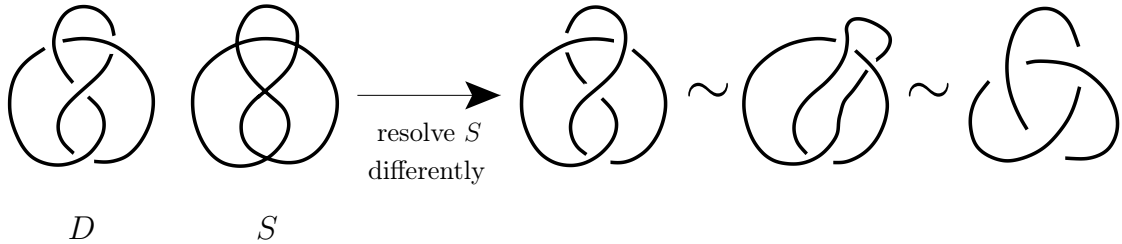}
\caption{The diagram $D$ is the usual diagram of the figure-eight knot $4_1$. As we illustrate, its shadow $S$ can be resolved into (a diagram of) the trefoil knot $3_1$.}
\label{fig:0001} 
\end{figure} 
% **************************************************************

\vglue 0.4 cm
\noindent{\bf Definition. }(Equivalent shadows). {\em Two shadows are {\em equivalent} if there is a (possibly orientation-reversing) self-homeomorphism of $\sphere$ that takes one to the other.}
\vglue 0.4 cm

Clearly two equivalent shadows have identical resolution properties.

A {\em component} of a shadow is a closed straight-ahead walk (this notion is recalled in Section~\ref{sec:gausscodes}). Thus a knot shadow has exactly one component, and in general a shadow of a link with $\ell$ components itself has $\ell$ components. See Figure~\ref{fig:exgaussB} for an example of a $3$-component shadow. This should not be confused with a graph-theoretical {\em connected component}: graph-theoretic connectivity is always referred to by saying that a shadow is {\em connected}.

We emphasize that a shadow is not necessarily a simple graph, as it may have loops and/or parallel edges. We assume that every component of a shadow has at least one vertex. This is a valid assumption for our purposes. Indeed, a vertexless knot shadow resolves only into the unknot, and if a multi-component link shadow $S$ has a vertexless component then $S$ resolves only into split links. In neither case does $S$ resolve into any of the knots or links under consideration (see Theorems~\ref{thm:knots} and~\ref{thm:links} below).

\subsection{Our main result}

Given a shadow $S$ and a link $L$, it is natural to ask whether $S$ resolves into $L$. This question and several variants have been thoroughly investigated in the literature~\cite{cantarella,endoitoh,evenzohar,hanaki2009,hanaki2010,hanaki2014,hanaki2015,hanaki2020,hannak,huhtaniyama,itotakimura,mahato,medina1,millett,smooth,ptaniyama}. The formal algorithmic question is the following:

\vglue 0.4cm 
\noindent{\bf Question. }{\em Let $L$ be a fixed link. Given a shadow $S$, what is the complexity of deciding whether $S$ resolves into $L$?}
\vglue 0.4 cm

We emphasize that this Question is only interesting if $L$ is not trivial. Indeed, every knot shadow can be resolved into the trivial knot (see for instance~\cite[Section 3.1]{adams}), and every shadow with $\ell>1$ components can be resolved into a split link that consists of $\ell$ trivial knots.

For a general fixed link $L$ it is in principle possible to approach this Question by brute force. If $S$ has $n$ vertices there are $2^n$ ways to resolve $S$ (as there are two possible ways to assign the over/under information at each vertex), and so in order to decide whether $S$ resolves into $L$ it suffices to check for each of these $2^n$ resolutions whether the resulting link is isotopic to $L$. 

Recent work of Lackenby shows that (for each fixed link $L$) determining whether a given diagram represents a link isotopic to $L$ lies in NP, and so it can be decided deterministically in exponential time~\cite{lackenby2} (see also~\cite{lackenby}). Therefore the brute-force approach above yields an exponential-time algorithm for our Question.

Our main result is the existence of a linear-time algorithm that decides whether an input shadow $S$ resolves into $L$, if $L$ is one of the knots in Figure~\ref{fig:0002} or one of the links in Figure~\ref{fig:0003}. We assume that the input shadow is given as a signed Gauss code for knots and a signed Gauss paragraph for links (we recall these notions in Section~\ref{sec:gausscodes}). Running time is measured in the number $|S|$ of vertices of $S$.

\begin{theorem}\label{thm:knots} 
For each $K \in \{3_1,4_1,5_1,5_2,6_2\}$ there is a linear-time algorithm that decides whether a knot shadow $S$ resolves into $K$. 
\end{theorem}

% **************************************************************
\begin{figure}[htbp] 
\def\ta#1{{\Scale[2.0]{#1}}} 
\def\tb#1{{\Scale[6.5]{#1}}} 
\def\somea{{\Scale[4.8]{\text{\rm (a)}}}} 
\def\someb{{\Scale[4.8]{\text{\rm (b)}}}} 
\def\somec{{\Scale[4.8]{\text{\rm (c)}}}} 
\def\somed{{\Scale[4.8]{\text{\rm (d)}}}} 
\def\somee{{\Scale[4.8]{\text{\rm (e)}}}} 
\def\somef{{\Scale[4.8]{\text{\rm (f)}}}} 
\def\pluschords{{\Scale[4.4]{${\text{\rm plus all its possible chords}}$}}} 
\centering 
\scalebox{0.25}{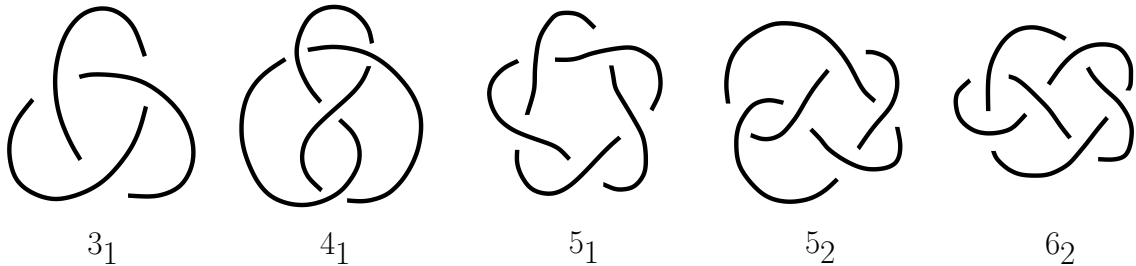}
\caption{The five knots in Theorem~\ref{thm:knots}.}
\label{fig:0002} 
\end{figure} 
% **************************************************************

\begin{theorem}\label{thm:links} 
For each $L \in \{L2a1,L4a1,L5a1,L6n1\}$ there is a linear-time algorithm that decides whether a link shadow $S$ resolves into $L$.
\end{theorem}

% **************************************************************
\begin{figure}[htbp] 
\def\ta#1{{\Scale[2.0]{#1}}} 
\def\tb#1{{\Scale[6.5]{#1}}} 
\def\somea{{\Scale[4.8]{\text{\rm (a)}}}} 
\def\someb{{\Scale[4.8]{\text{\rm (b)}}}} 
\def\somec{{\Scale[4.8]{\text{\rm (c)}}}} 
\def\somed{{\Scale[4.8]{\text{\rm (d)}}}} 
\def\somee{{\Scale[4.8]{\text{\rm (e)}}}} 
\def\somef{{\Scale[4.8]{\text{\rm (f)}}}} 
\def\pluschords{{\Scale[4.4]{${\text{\rm plus all its possible chords}}$}}} 
\centering 
\scalebox{0.25}{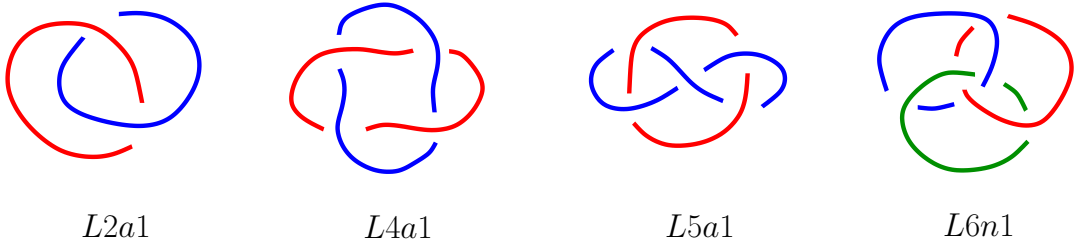}
\caption{The four links  in Theorem~\ref{thm:links}.}
\label{fig:0003} 
\end{figure} 
% **************************************************************

\subsection{Strategy and related work}\label{sub:stra}

The workhorse behind the proofs of Theorems~\ref{thm:knots} and~\ref{thm:links} is a characterization of the shadows that resolve into a prescribed link. For the knots in Theorem~\ref{thm:knots} the relevant characterizations are stated for {\em prime} knot shadows (a notion we recall in Section~\ref{sec:primeshadows}), so we first need to decompose an arbitrary knot shadow into prime shadows. For $L5a1$ we similarly need to reduce any given shadow to a prime $2$-component link shadow.

Taniyama found the characterizations that we use for the knots $3_1, 4_1, 5_1, 5_2$ and for the two-component links $L2a1, L4a1, L5a1$, as by-products of his work on a partial order on knots~\cite{taniyamaknots} and links~\cite{taniyamalinks}. More recently Takimura~\cite{takimura} found a characterization for the knot $6_2$. Finally,~\cite{ars} gives a characterization for the three-component link $L6n1$, the only non-alternating three-component link with six crossings.

After the appropriate prime decomposition, when one is needed, the knot cases and the $L5a1$ case reduce to recognizing a small number of exceptional shadow families (the other multi-component link cases reduce to simpler tests). As we shall see, all the required tests can be performed in linear time.

\section{Gauss codes and paragraphs}\label{sec:gausscodes}

As we mentioned above, we assume that the input shadow in Theorem~\ref{thm:knots} is given by a signed Gauss code, and the input shadow in Theorem~\ref{thm:links} is given by a signed Gauss paragraph. Our aim in this section is to recall these standard knot-theoretical notions. We assume that an $n$-vertex input shadow has vertices labelled $1,\ldots,n$.

\subsection{Signed Gauss codes of knot shadows}\label{sub:gaussknot}

Let $S$ be a knot shadow with $n$ vertices. We choose any non-vertex point $p$ of $S$ as a starting basepoint, choose a traversal direction, and traverse the shadow in a straight-ahead manner until we reach $p$ again. Here ``straight-ahead'' means that whenever we reach a vertex we continue the traversal following the edge that is opposite to the edge from which we reached the vertex. See Figure~\ref{fig:exgaussA} for an illustration.

% **************************************************************
\begin{figure}[htbp] 
\def\ta#1{{\Scale[1.6]{#1}}} 
\def\tb#1{{\Scale[6.5]{#1}}} 
\def\somea{{\Scale[4.8]{\text{\rm (a)}}}} 
\def\someb{{\Scale[4.8]{\text{\rm (b)}}}} 
\def\somec{{\Scale[4.8]{\text{\rm (c)}}}} 
\def\somed{{\Scale[4.8]{\text{\rm (d)}}}} 
\def\somee{{\Scale[4.8]{\text{\rm (e)}}}} 
\def\somef{{\Scale[4.8]{\text{\rm (f)}}}} 
\def\pluschords{{\Scale[4.4]{${\text{\rm plus all its possible chords}}$}}} 
\centering 
\scalebox{0.25}{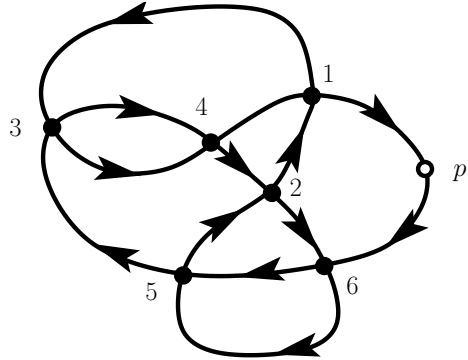}
\caption{A signed Gauss code of this knot shadow is $6^+\,5^-\,3^-\,4^+\,2^+\,6^-\,5^+\,2^-$ $1^-\,3^+\,4^-\,1^+$.}
\label{fig:exgaussA} 
\end{figure} 
% **************************************************************

We record the vertices in the order in which we encounter them, assigning a sign to each occurrence with the convention illustrated in Figure~\ref{fig:gauss1}: an occurrence of $v$ is recorded as $v^+$ when we reach $v$ via its {\em positive} incoming edge, and as $v^-$ if we reach it via its negative incoming edge. The word $\Gamma$ recorded in this way is a {\em signed Gauss code} of $S$. It is not unique: it depends on the choice of the basepoint $p$, of the traversal direction, and of the orientation of $\sphere$.

% **************************************************************
\begin{figure}[htbp] 
\def\ta#1{{\Scale[1.6]{#1}}} 
\def\tb#1{{\Scale[6.5]{#1}}} 
\def\somea{{\Scale[4.8]{\text{\rm (a)}}}} 
\def\someb{{\Scale[4.8]{\text{\rm (b)}}}} 
\def\somec{{\Scale[4.8]{\text{\rm (c)}}}} 
\def\somed{{\Scale[4.8]{\text{\rm (d)}}}} 
\def\somee{{\Scale[4.8]{\text{\rm (e)}}}} 
\def\somef{{\Scale[4.8]{\text{\rm (f)}}}} 
\def\posi{\scalebox{4}{\small\rmfamily positive incoming edge}}
\def\nega{\scalebox{4}{\small\rmfamily negative incoming edge}}
\centering 
\scalebox{0.25}{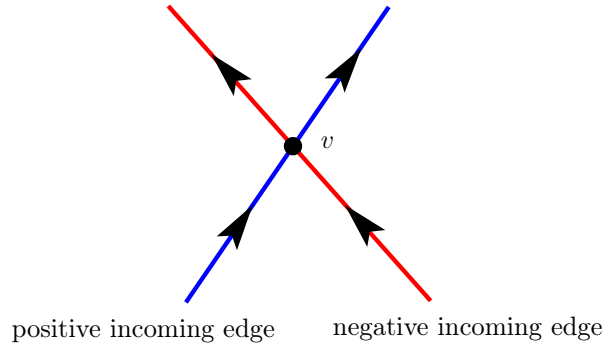}
\caption{The sign convention for the incoming edges at a vertex $v$.}
\label{fig:gauss1} 
\end{figure} 
% **************************************************************

Thus a Gauss code has $2n$ symbols, and each vertex occurs once with each sign.

For a Gauss code $\Gamma=a_1^{\varepsilon_1}\cdots a_{2n}^{\varepsilon_{2n}}$ (where $\varepsilon_i \in \{+,-\}$ for $i=1,\ldots,2n$), changing the basepoint gives a {\em cyclic translation} of $\Gamma$, reversing the traversal gives $\rev{\Gamma}:=a_{2n}^{\varepsilon_{2n}}\cdots a_1^{\varepsilon_1}$, and reversing the orientation of $\sphere$ gives $\Neg{\Gamma}:=a_1^{-\varepsilon_1}\cdots a_{2n}^{-\varepsilon_{2n}}$.

\vglue 0.4 cm
\noindent{\bf Definition. }(Congruent signed Gauss codes). {\em Two signed Gauss codes $\Gamma$ and $\Gamma'$ are {\em congruent} if $\Gamma'$ can be obtained from $\Gamma$ by some sequence of cyclic translation, reversal, and negation operations, possibly also using a relabelling of the vertices.}
\vglue 0.4 cm

The following standard fact characterizes equivalent shadows; see Carter~\cite{carter1991} or Erickson~\cite{erickson}.

\begin{proposition}\label{pro:carter}
Let $S$ and $S'$ be knot shadows, and let $\Gamma$ and $\Gamma'$ be signed Gauss codes of $S$ and $S'$, respectively. Then $S$ and $S'$ are equivalent if and only if $\Gamma$ and $\Gamma'$ are congruent.
\end{proposition}

Suppose now that $S$ is a shadow of a link with $\ell > 1$ components, and let $W_1,\ldots,W_\ell$ be the components of $S$. For each $W_i$, choose a basepoint and direction and record as above the resulting signed Gauss word $\Gamma_i$, which we call a {\em signed Gauss word of $W_i$}. The unordered collection $\mathbf{\Gamma}=\{\Gamma_1,\ldots,\Gamma_\ell\}$ is a {\em signed Gauss paragraph} of $S$. As in signed Gauss codes, each vertex occurs exactly twice in the paragraph, once with each sign. See Figure~\ref{fig:exgaussB} and its caption for an example.

% **************************************************************
\begin{figure}[htbp] 
\def\ta#1{{\Scale[2.0]{#1}}} 
\def\tb#1{{\Scale[6.5]{#1}}} 
\def\somea{{\Scale[4.8]{\text{\rm (a)}}}} 
\def\someb{{\Scale[4.8]{\text{\rm (b)}}}} 
\def\somec{{\Scale[4.8]{\text{\rm (c)}}}} 
\def\somed{{\Scale[4.8]{\text{\rm (d)}}}} 
\def\somee{{\Scale[4.8]{\text{\rm (e)}}}} 
\def\somef{{\Scale[4.8]{\text{\rm (f)}}}} 
\def\pluschords{{\Scale[4.4]{${\text{\rm plus all its possible chords}}$}}} 
\centering 
\scalebox{0.25}{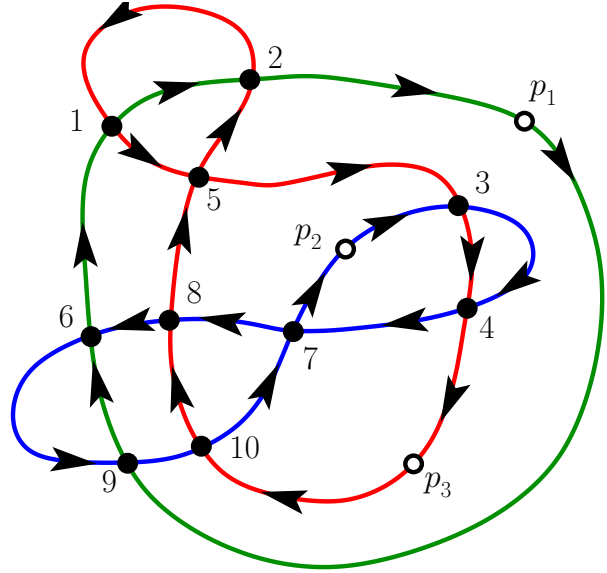}
\caption{A link shadow with three components, coloured green, blue, and red. If we traverse the green component starting at $p_1$ in the given direction, the blue component starting at $p_2$ in the given direction, and the red component starting at $p_3$ in the given direction, we obtain the signed Gauss paragraph $\{9^-\,6^+\,1^-\,2^+,\,\,3^-\,4^+\,$ $7^-\,8^-\,6^-\,9^+\,10^+\,7^+,$ $  \, 10^-\,8^+\,5^-\,2^-\,1^+\,5^+\,3^+\,4^-\}$.}
\label{fig:exgaussB}
\end{figure} 
% **************************************************************

We define congruence for Gauss paragraphs similarly to congruence for Gauss codes. Reversing the traversal of a component $W_i$ reverses $\Gamma_i$, and at every vertex shared with another component both occurrences change sign.

To capture this, a {\em single reversal} at $\Gamma_i$ consists of reversing the order of $\Gamma_i$ and, for every vertex that occurs exactly once in $\Gamma_i$, changing the signs of both occurrences of that vertex in the whole Gauss paragraph.

We thus obtain the following notion of congruence for signed Gauss paragraphs.

\vglue 0.4 cm 
\noindent{\bf Definition. }(Congruent signed Gauss paragraphs). {\em Two signed Gauss paragraphs $\BG$ and $\BG'$ are {\em congruent} if $\BG'$ can be obtained from $\BG$ by some sequence of (i) cyclic translations of signed Gauss words; (ii) single reversals; and (iii) the negation of all the signed Gauss words, possibly also using a relabelling of the vertices.}
\vglue 0.4 cm

As with Proposition~\ref{pro:carter}, for a proof of the next statement we refer the reader to Carter~\cite{carter1991} or Erickson~\cite{erickson}.

\begin{proposition}\label{pro:carterlinks}
Let $S$ and $S'$ be connected link shadows with the same number of components, and let $\BG$ and $\BG'$ be signed Gauss paragraphs of $S$ and $S'$, respectively. Then $S$ and $S'$ are equivalent if and only if $\BG$ and $\BG'$ are congruent.
\end{proposition}

\vglue 0.4 cm
\noindent{\bf Remark. }(Omitting ``signed'' from Gauss codes and paragraphs). {\em For brevity, whenever we mention a Gauss code (respectively, a Gauss paragraph), we implicitly mean a {\em signed} Gauss code (respectively, {\em signed} Gauss paragraph).}
\vglue 0.4 cm

\section{Prime shadows and prime decomposition}\label{sec:primeshadows}

As we shall see, the characterizations that establish when a shadow resolves into a given knot are stated for {\em prime} knot shadows. Also, one of the characterizations for links (namely the one that establishes when a link shadow resolves into $L5a1$) is stated for {\em prime} link shadows. 

The primeness of a shadow is in turn defined in terms of the connected sum of shadows. These two notions (primeness and the connected sum of shadows) are natural analogues of prime links and of the connected sum of links. We recall that a link $L$ is {\em prime} if whenever $L$ is a connected sum $L_1\# L_2$ one of $L_1$ and $L_2$ is the trivial knot. 

Let $S_1$ and $S_2$ be link (perhaps knot) shadows. Let $e_1$ and $e_2$ be edges in $S_1$ and $S_2$, respectively. As we illustrate in Figure~\ref{fig:consum}, we cut open $S_1$ and $S_2$ at $e_1$ and $e_2$, respectively, and reconnect the four resulting loose ends without introducing any crossings. We use $S_1\# S_2$ to denote the resulting shadow, and say that it is a {\em connected sum} of $S_1$ and $S_2$. We write $S_1\#\cdots\# S_k$ for any shadow obtained by iterating this operation.

% **************************************************************
\begin{figure}[htbp] 
\def\ta#1{{\Scale[3.0]{#1}}} 
\def\tb#1{{\Scale[3.5]{#1}}} 
\def\eone{{\ta{e_{{}_{1}}}}}
\def\etwo{{\ta{e_{{}_{2}}}}}
\def\eseuno{{\tb{S_{{}_{1}}}}}
\def\esedos{{\tb{S_{{}_{2}}}}}
\def\esesola{{\tb{S}}}
\def\pluschords{{\Scale[4.4]{${\text{\rm plus all its possible chords}}$}}} 
\centering 
\scalebox{0.15}{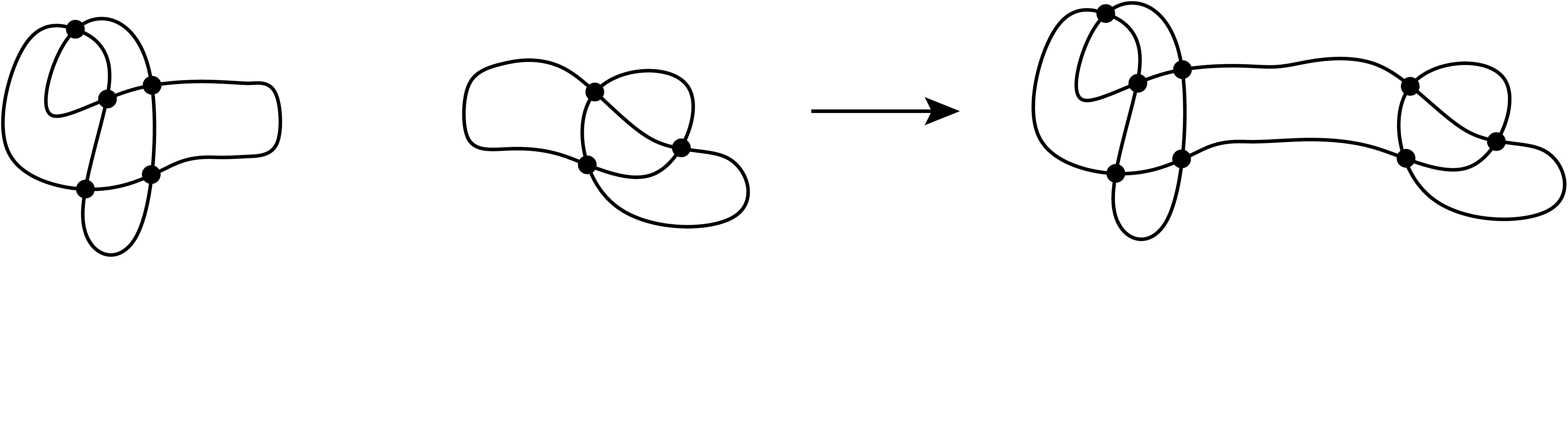}
\caption{If we cut open the edges $e_1$ and $e_2$ and glue the severed ends as shown, we merge $S_1$ and $S_2$ into a single shadow $S$, a {\em connected sum} $S_1\# S_2$ of $S_1$ and $S_2$.}
\label{fig:consum} 
\end{figure} 
% **************************************************************

\vglue 0.4 cm
\noindent{\bf Definition. }(Prime shadows). {\em A connected shadow $S$ with at least two vertices is {\em prime} if it is not a connected sum of two shadows each with at least one vertex.}
\vglue 0.4 cm

It is easy to verify that we may equivalently define prime shadows as follows.

\begin{observation}[Two equivalent conditions for a shadow to be prime]\label{obs:equivtoprime}
A connected shadow with at least two vertices is prime if and only if it has no $2$-edge cut. Equivalently, this holds if and only if its underlying multigraph is $3$-edge-connected.
\end{observation}

We note that since a connected $4$-regular multigraph is Eulerian and hence has no cut-edge, the last equivalence in this observation simply says that there is no edge cut of size less than three.

As we mentioned above, the notion of primeness in shadows and the connected sum of shadows are natural analogues of their counterparts in links. Indeed, resolving $S_1\# S_2$ evidently yields a connected sum of a resolution of $S_1$ and a resolution of $S_2$. Thus the following is an immediate consequence of the uniqueness of the decomposition of a knot into prime knots (Schubert~\cite{schubert}), together with the fact, noted above, that every knot shadow resolves into the trivial knot.

\begin{theorem}\label{thm:motiv} 
Let $K$ be a prime knot. If a knot shadow $S$ is a connected sum $S_1\# \cdots \# S_k$, then $S$ resolves into $K$ if and only if some $S_i$ resolves into $K$.
\end{theorem}

Focusing on knot shadows first, in the context of Theorem~\ref{thm:knots} we are interested in deciding, given a knot shadow $S$, whether $S$ resolves into a knot $K$ in $\{3_1,4_1,5_1,5_2,6_2\}$. All these knots are prime. Therefore if $S$ is a connected sum $S_1\# S_2$ then $S$ resolves into $K$ if and only if $S_1$ or $S_2$ resolves into $K$.

Loosely speaking, this means that in order to verify if $S$ resolves into $K$ we may recursively ``reduce'' $S$ along $2$-edge cuts, until we have decomposed $S$ into prime shadows with the property that $S$ resolves into $K$ if and only if at least one of these prime shadows resolves into $K$. The following lemma shows that this reduction can be performed in linear time.

\begin{lemma}\label{lem:primesknots}
Given the Gauss code of a knot shadow $S$ on $n$ vertices, one can compute, in $O(n)$ time, the Gauss codes of a possibly empty collection of prime knot shadows $S_1,\ldots,S_k$, such that
$$
|S_1|+\cdots+|S_k|\leq n,
$$
and with the property that for every nontrivial prime knot $K$
$$
S\text{ resolves into }K
\quad\Longleftrightarrow\quad
S_i\text{ resolves into }K\text{ for some }i.
$$
\end{lemma}

The proof of Lemma~\ref{lem:primesknots} is deferred to Section~\ref{sec:primeproofs}. We conclude the section with the following version for $2$-component links. Its proof is also deferred to Section~\ref{sec:primeproofs}.

%%%%%%%%%%%%%%%%%%%%%%%%%%%%%%%%%%%%%%%%%%%%%%%%%%%%%%%%%%%%%%%%%%%%%%%%%%
\begin{lemma}\label{lem:primeslinks}
Given the Gauss paragraph of a connected $2$-component link shadow $S$ on $n$ vertices, one can compute, in $O(n)$ time, the Gauss paragraph of a connected prime $2$-component link shadow $S'$ with $|S'|\le n$, and with the property that $S$ resolves into a prime link $L$ if and only if $S'$ resolves into $L$.
\end{lemma}

%%%%%%%%%%%%%%%%%%%%%%%%%%%%%%%%%%%%%%%%%%%%%%%%%%%%%%%%%%%%%%%%%%%%%%%%%%5

\section{The families of shadows involved in the characterizations}\label{sec:theshadows}

Several of the characterizations of the shadows that resolve into the knots in Theorem~\ref{thm:knots} or into the links in Theorem~\ref{thm:links} involve excluding one or more families of shadows. Our purpose in this section is to recall these families.

\subsection{Torus shadows}

Several characterization lemmas involve the one-parameter family of shadows $T(2,p)$. We call the shadows in this family {\em torus shadows} because, as we illustrate in Figure~\ref{fig:torusknotlink} for the cases $p=7$ and $p=8$, for each integer $p \ge 1$ the shadow $T(2,p)$ is the shadow of the usual diagram of the $(2,p)$-torus knot if $p$ is odd, or of the $(2,p)$-torus link if $p$ is even. Thus, regardless of the parity of $p$, the torus shadow $T(2,p)$ has $p$ vertices. 

% **************************************************************
\begin{figure}[htbp] 
\def\ta#1{{\Scale[3.6]{#1}}} 
\def\tb#1{{\Scale[4.0]{#1}}} 
\def\tc#1{{\Scale[3.8]{#1}}} 
\def\somea{{\Scale[4.8]{\text{\rm (a)}}}} 
\def\somea{{\Scale[4.8]{\text{\rm (a)}}}} 
\def\someb{{\Scale[4.8]{\text{\rm (b)}}}} 
\def\somec{{\Scale[4.8]{\text{\rm (c)}}}} 
\def\somed{{\Scale[4.8]{\text{\rm (d)}}}} 
\def\somee{{\Scale[4.8]{\text{\rm (e)}}}} 
\def\somef{{\Scale[4.8]{\text{\rm (f)}}}} 
\def\tsa{\scalebox{9}{\small\rmfamily (Knot shadow)}}
\def\ltsa{\scalebox{9}{\small\rmfamily ($2$-component link shadow)}}
\centering 
\scalebox{0.12}{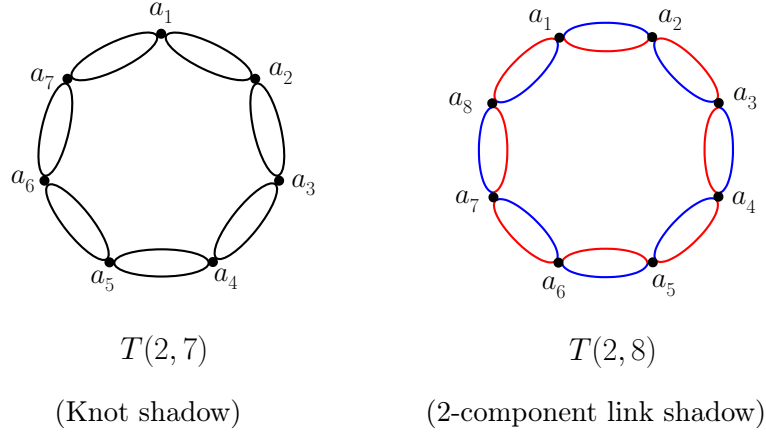}
\caption{The torus shadow $T(2,7)$ is the shadow of the usual diagram of the $(2,7)$-torus knot ($7_1$ in Rolfsen's table), and $T(2,8)$ is the shadow of the usual diagram of the $2$-component $(2,8)$-torus link (L8a14 in Thistlethwaite's link table). One of the components of $T(2,8)$ is coloured blue, and the other component is coloured red.}
\label{fig:torusknotlink} 
\end{figure} 
% **************************************************************

For $p=1$, the shadow $T(2,1)$ consists of one vertex and two loop edges. For $p=2$, its two vertices are joined by four parallel edges. For $p\geq3$, the vertices of $T(2,p)$ can be labelled $a_1,\ldots,a_p$ cyclically so that, for each $i$, the vertices $a_i$ and $a_{i+1}$ are joined by two parallel edges, where indices are read modulo $p$.

\subsection{Pretzel shadows}

As we illustrate in Figure~\ref{fig:pretzels} for the case $p_1=7, p_2=3, p_3=5$, for any odd positive integers $p_1, p_2,p_3$ the {\em pretzel shadow} $S(p_1,p_2,p_3)$ is the shadow of the usual diagram of the pretzel knot $P(p_1,p_2,p_3)$.

% **************************************************************
\begin{figure}[htbp] 
\def\ta#1{{\Scale[5.0]{#1}}} 
\def\tb#1{{\Scale[3.6]{#1}}} 
\def\td#1{{\Scale[4.6]{#1}}} 
\def\Ap#1{{\Scale[4.6]{a_{{}_{#1}}}}} 
\def\Bp#1{{\Scale[4.6]{b_{{}_{#1}}}}} 
\def\Cp#1{{\Scale[4.6]{c_{{}_{#1}}}}} 
\def\somea{{\Scale[4.8]{\text{\rm (a)}}}} 
\def\someb{{\Scale[4.8]{\text{\rm (b)}}}} 
\def\somec{{\Scale[4.8]{\text{\rm (c)}}}} 
\def\somed{{\Scale[4.8]{\text{\rm (d)}}}} 
\def\somee{{\Scale[4.8]{\text{\rm (e)}}}} 
\def\somef{{\Scale[4.8]{\text{\rm (f)}}}} 
\def\psa{\scalebox{11}{\small\rmfamily Pretzel shadow}}
\centering 
\scalebox{0.1}{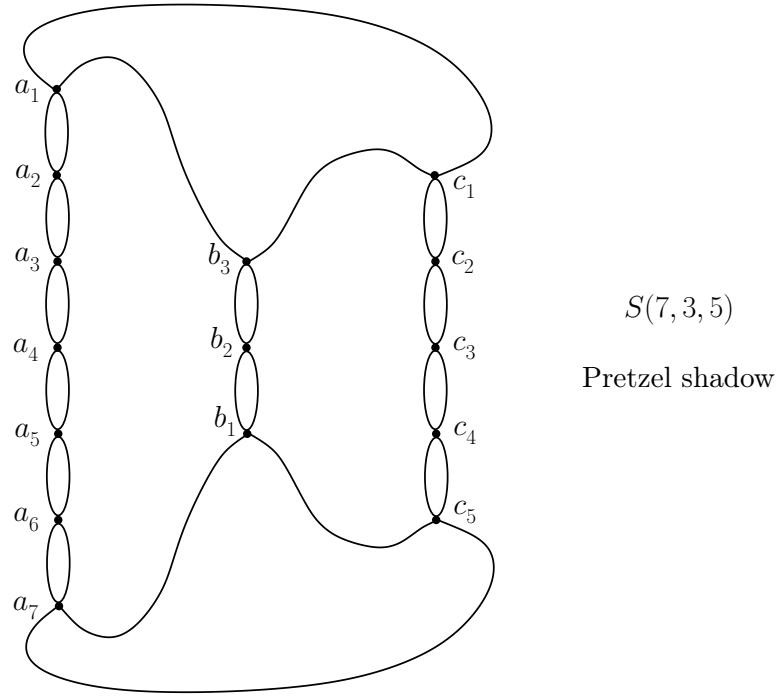}
\caption{The pretzel shadow $S(7,3,5)$ is the shadow of the usual diagram of the pretzel knot $P(7,3,5)$.}
\label{fig:pretzels} 
\end{figure} 
% **************************************************************

As in Figure~\ref{fig:pretzels}, the vertices of $S(p_1,p_2,p_3)$ are naturally partitioned into three ``chains'', a chain $A=(a_1,\ldots,a_{p_1})$, a chain $B=(b_1,\ldots,b_{p_2})$, and a chain $C=(c_1,\ldots,c_{p_3})$. For $i=2,\ldots,p_1-1$ vertex $a_i$ is adjacent via two parallel edges to each of $a_{i-1}$ and $a_{i+1}$, and similarly for $b_i$ for $i=2,\ldots,p_2-1$ and each $c_i$ for $i=2,\ldots,p_3-1$ (in a degenerate case in which for instance $p_3=1$, there is only one vertex in the $C$ chain and so there are no such parallel edges).  Finally, $a_1, b_{p_2}$, and $c_1$ are adjacent to each other, and $a_{p_1}, b_{1}$, $c_{p_3}$ are adjacent to each other.

The labelling of the vertices $b_1,\ldots,b_{p_{{}_2}}$ may seem counterintuitive: in Figure~\ref{fig:pretzels} the indices of the $a_i$ and of the $c_i$ increase downwards, whereas the indices of the $b_i$ increase upwards. As it happens, this choice turns out to be very convenient when we deal with the Gauss code of these shadows.

\subsection{Nutcracker shadows}

The last infinite family of shadows involved in the characterizations is that of the {\em nutcracker shadows}. As we illustrate in Figure~\ref{fig:nuts} for the case $s=2, t=3$, for any two positive integers $s$ and $t$ the shadow $N(2s,2t)$ has $2s+2t$ vertices, with $2s$ of them forming a chain $A=(a_1,\ldots,a_{2s})$ and the other $2t$ vertices forming another chain $B=(b_1,\ldots,b_{2t})$. As with pretzel shadows, for $i=1,\ldots,2s-1$ vertices $a_i$ and $a_{i+1}$ are joined by two parallel edges, and for $i=1,\ldots,2t-1$ vertices $b_i$ and $b_{i+1}$ are joined by two parallel edges. Finally, each of $a_1$ and $a_{2s}$ is joined by an edge to each of $b_1$ and $b_{2t}$.

% **************************************************************
\begin{figure}[htbp] 
\def\ta#1{{\Scale[5.0]{#1}}} 
\def\tb#1{{\Scale[3.6]{#1}}} 
\def\td#1{{\Scale[4.6]{#1}}} 
\def\Ap#1{{\Scale[4.6]{b_{{}_{#1}}}}} 
\def\Bp#1{{\Scale[4.6]{b_{{}_{#1}}}}} 
\def\Cp#1{{\Scale[4.6]{a_{{}_{#1}}}}} 
\def\somea{{\Scale[4.8]{\text{\rm (a)}}}} 
\def\someb{{\Scale[4.8]{\text{\rm (b)}}}} 
\def\somec{{\Scale[4.8]{\text{\rm (c)}}}} 
\def\somed{{\Scale[4.8]{\text{\rm (d)}}}} 
\def\somee{{\Scale[4.8]{\text{\rm (e)}}}} 
\def\somef{{\Scale[4.8]{\text{\rm (f)}}}} 
\def\psa{\scalebox{11}{\small\rmfamily Nutcracker shadow}}
\centering 
\scalebox{0.1}{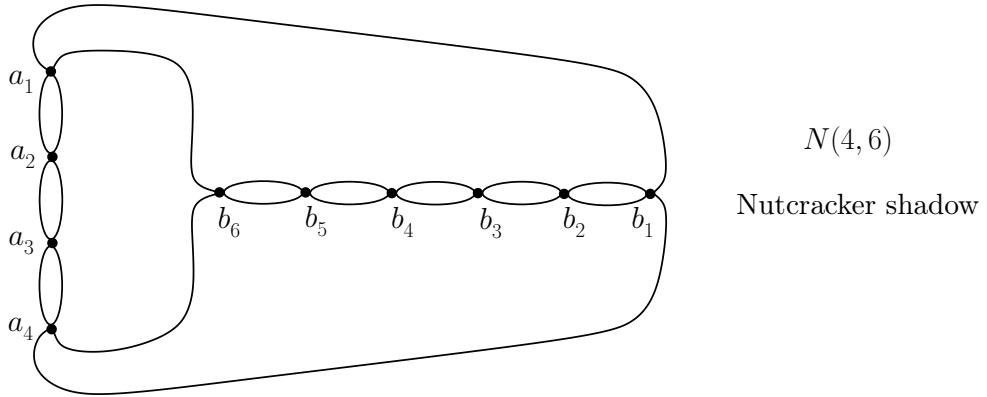}
\caption{The nutcracker shadow $N(4,6)$.}
\label{fig:nuts} 
\end{figure} 
% **************************************************************

\subsection{Two sporadic shadows}

The last characterization lemma for knots states which shadows resolve into the knot $6_2$, and involves the three infinite families of shadows identified above, plus two very specific shadows, one with six vertices and one with eight vertices.

As we illustrate in Figure~\ref{fig:6_and_8}, we let $S(6_3)$ denote the shadow of the usual diagram of the knot $6_3$, and we let $S(8_{18})$ be the shadow of the usual diagram of the knot $8_{18}$.

% **************************************************************
\begin{figure}[htbp] 
\def\ta#1{{\Scale[3.0]{#1}}} 
\def\tb#1{{\Scale[6.5]{#1}}} 
\def\somea{{\Scale[4.8]{\text{\rm (a)}}}} 
\def\someb{{\Scale[4.8]{\text{\rm (b)}}}} 
\def\somec{{\Scale[4.8]{\text{\rm (c)}}}} 
\def\somed{{\Scale[4.8]{\text{\rm (d)}}}} 
\def\somee{{\Scale[4.8]{\text{\rm (e)}}}} 
\def\somef{{\Scale[4.8]{\text{\rm (f)}}}} 
\def\pluschords{{\Scale[4.4]{${\text{\rm plus all its possible chords}}$}}} 
\centering 
\scalebox{0.15}{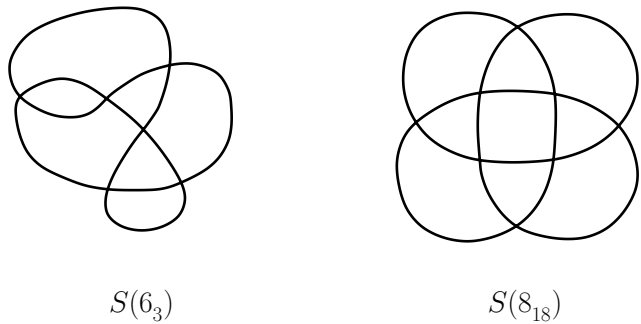}
\caption{The shadows $S(6_3)$ and $S(8_{18})$ do not resolve into the knot $6_2$ (Lemma~\ref{lem:62}).}
\label{fig:6_and_8} 
\end{figure} 
% **************************************************************

%%%%%%%%%%%%%%%%%%%%%%%%%%%%%%%%%%%%%%%%%%%%%%%%%%%%%%%%%%%%%%%%%%%%%%%%%%5

\subsection{The recognition lemmas}\label{sub:testingks}

We finish this section by stating the existence of $O(n)$ time algorithms to test whether a given shadow $S$ on $n$ vertices is equivalent to one of the shadow families discussed above. In order to keep the discussion going, we defer the proofs of these lemmas to Section~\ref{sec:equivtests}.

\begin{lemma}\label{lem:testknot}
Given the Gauss code of a knot shadow $S$ on $n$ vertices, one can answer in $O(n)$ time each of the following queries:
\begin{enumerate}

\item[(i)] Is $S$ equivalent to a torus shadow?

\item[(ii)] Is $S$ equivalent to a pretzel shadow?

\item[(iii)] Is $S$ equivalent to a nutcracker shadow?

\end{enumerate}

\noindent Moreover, for each fixed shadow $S_0$ (we have in mind especially $S(6_3)$ and $S(8_{18})$)  one can answer in $O(n)$ time the following query:

\begin{enumerate}

\item[(iv)] Is $S$ equivalent to $S_0$?

\end{enumerate}

\end{lemma}

For shadows of $2$-component links we have the following (easier and shorter) version of Lemma~\ref{lem:testknot}.

\begin{lemma}\label{lem:test2link}
Given the Gauss paragraph of a $2$-component link shadow $S$ on $n$ vertices, one can answer in $O(n)$ time each of the following queries:
\begin{enumerate}

\item[(i)] Do the two words that form the Gauss paragraph of $S$ share at least $4$ vertices?

\item[(ii)] Is $S$ equivalent to a torus shadow?

\end{enumerate}

\end{lemma}

Finally, for shadows of $3$-component links we only need the following.

\begin{lemma}\label{lem:test3link}
Given the Gauss paragraph of a $3$-component link shadow $S$ on $n$ vertices, one can answer in $O(n)$ time the following query. Does every pair of component words share at least one vertex?
\end{lemma}

%%%%%%%%%%%%%%%%%%%%%%%%%%%%%%%%%%%%%%%%%%%%%%%%%%%%%%%%%%%%%%%%%%%%%%%%%%5
%%%%%%%%%%%%%%%%%%%%%%%%%%%%%%%%%%%%%%%%%%%%%%%%%%%%%%%%%%%%%%%%%%%%%%%%%%5
%%%%%%%%%%%%%%%%%%%%%%%%%%%%%%%%%%%%%%%%%%%%%%%%%%%%%%%%%%%%%%%%%%%%%%%%%%%

\section{Proofs of Theorems~\ref{thm:knots} and~\ref{thm:links}}\label{sec:mainproofs}

We start by recalling the characterizations that are at the heart of the proofs of Theorems~\ref{thm:knots} and~\ref{thm:links}. The first four involve knot shadows, and are due to Taniyama~\cite{taniyamaknots}. The fifth one was established more recently by Takimura~\cite{takimura}.

The statements below are convenient (for our purposes) reformulations of Taniyama's results, obtained from the statements and proofs of the corresponding theorems.

\subsection{The characterizations for the knots $3_1, 4_1, 5_1, 5_2$, and $6_2$}

We emphasize that Lemmas~\ref{lem:31}--\ref{lem:52} were not stated in this form in~\cite{taniyamaknots}, but it is straightforward to derive these statements from the respective theorems (and their proofs) in~\cite{taniyamaknots}. Moreover, these four lemmas are also stated in essentially this form in the review of Taniyama's work given by Takimura in~\cite[Theorem 1]{takimura}.

\begin{lemma}[Follows from {\cite[Theorem 1]{taniyamaknots}}]\label{lem:31} 
Every prime knot shadow resolves into the trefoil knot $3_1$.
\end{lemma}

\begin{lemma}[Follows from {\cite[Theorem 2]{taniyamaknots}}]\label{lem:41}
A prime knot shadow with $p \ge 4$ vertices resolves into the figure-eight knot $4_1$ if and only if it is not equivalent to $T(2,p)$.
\end{lemma}

We note that since in a knot torus shadow $T(2,p)$ the parameter $p$ is necessarily odd, Lemma~\ref{lem:41} in particular implies that {\em any} prime knot shadow with an even number $p\ge 4$ of vertices {\em always} resolves into $4_1$. 

\begin{lemma}[Follows from {\cite[Theorem 4]{taniyamaknots}}]\label{lem:51}
A prime knot shadow resolves into $5_1$ if and only if it is equivalent to neither $S(p_1,p_2,p_3)$, for any odd positive integers $p_1,p_2,p_3$, nor to $N(2s,2t)$, for any positive integers $s,t$. 
\end{lemma}

\begin{lemma}[Follows from {\cite[Theorem 3]{taniyamaknots}}]\label{lem:52}
A prime knot shadow resolves into $5_2$ if and only if it is equivalent to neither $T(2,p)$, for any odd $p\ge 3$, nor to $N(2,2)$.
\end{lemma}

\begin{lemma}[{\cite[Theorem 2]{takimura}}]\label{lem:62} 
A prime knot shadow fails to resolve into $6_2$ if and only if it is equivalent to one of the following: $T(2,p)$ for some odd $p\ge 3$; $S(p_1,p_2,p_3)$ for some positive odd integers $p_1,p_2,p_3$; $N(2s,2t)$ for some positive integers $s,t$; $S(6_3)$; or $S(8_{18})$.
\end{lemma}

\subsection{The characterizations for the links $L2a1, L4a1, L5a1$, and $L6n1$}

The next three lemmas are also due to Taniyama, this time for two-component links. As above, we emphasize that Lemmas~\ref{lem:L2a1}--\ref{lem:L5a1} were not stated in this form in~\cite{taniyamalinks}, but it is straightforward to derive these statements from the respective theorems in~\cite{taniyamalinks}.

The first two characterizations for link shadows do not involve any infinite families:

\begin{lemma}[Follows from {\cite[Theorem 1]{taniyamalinks}}]\label{lem:L2a1} 
A two-component link shadow resolves into $L2a1$ if and only if it is connected. 
\end{lemma}

\begin{lemma}[Follows from {\cite[Theorem 2]{taniyamalinks}}]\label{lem:L4a1} 
A two-component link shadow, whose two components we denote $S_1$ and $S_2$, resolves into $L4a1$ if and only if $S_1$ and $S_2$ share at least $4$ vertices. 
\end{lemma}

The characterization for those link shadows that resolve into the Whitehead link $L5a1$ involves the family of (link) torus shadows introduced above.

\begin{lemma}[Follows from {\cite[Theorem 3]{taniyamalinks}}]\label{lem:L5a1} 
A two-component prime link shadow resolves into $L5a1$ if and only if it is not equivalent to $T(2,p)$ for any even $p \ge 2$.
\end{lemma}

Finally, we recall the characterization of the three-component link $L6n1$~\cite{ars}, the only non-alternating three-component link with six crossings in the Thistlethwaite table. If $S$ is a three-component link shadow with components $S_1, S_2, S_3$, we say that $S$ is {\em pairwise crossing} if for any distinct $i,j\in\{1,2,3\}$ the components $S_i$ and $S_j$ are not disjoint.

\begin{lemma}[{\cite{ars}}]\label{lem:L6n1} 
A three-component link shadow resolves into $L6n1$ if and only if it is pairwise crossing. 
\end{lemma}

% ******************************************************************************************
% ******************************************************************************************
% ******************************************************************************************
% ******************************************************************************************
% ******************************************************************************************
% ******************************************************************************************
% ******************************************************************************************

%%%%%%%%%%%%%%%%%%%%%%%%%%%%%%%%%%%%%%%%%%%%%%%%%%%%%%%%%%%%%%%%%%%%%%%%%%%
%%%%%%%%%%%%%%%%%%%%%%%%%%%%%%%%%%%%%%%%%%%%%%%%%%%%%%%%%%%%%%%%%%%%%%%%%%%

\subsection{Proofs of Theorems~\ref{thm:knots} and~\ref{thm:links}}

\begin{proof}[Proof of Theorem~\ref{thm:knots}]
Apply Lemma~\ref{lem:primesknots} to $S$, obtaining in $O(n)$ time the Gauss codes of prime shadows $S_1,\ldots,S_k$. Let $n_i:=|S_i|$ for $i=1,\ldots,k$. Thus 
$$ n_1+\cdots+n_k\leq n. $$
Moreover, since each of the knots
$$ 3_1,4_1,5_1,5_2,6_2 $$
is a nontrivial prime knot, Lemma~\ref{lem:primesknots} implies that, for each $K$ in this set,
$$S\text{ resolves into }K \quad\Longleftrightarrow\quad S_i\text{ resolves into }K\text{ for some }i.$$

Thus it suffices, for each $i$, to decide in $O(n_i)$ time whether $S_i$ resolves into $K$.

If $K=3_1$, Lemma~\ref{lem:31} says that every $S_i$ resolves into $3_1$, so $S$ resolves into $3_1$ if and only if $k\ge1$.

If $K=4_1$, shadows $S_i$ with $n_i<4$ do not resolve into $4_1$. On the other hand, for $n_i\ge 4$ Lemma~\ref{lem:41} and Lemma~\ref{lem:testknot}(i) decide in $O(n_i)$ time whether $S_i$ resolves into $4_1$, namely by testing whether $S_i$ is a torus shadow.

For $K=5_1$, Lemma~\ref{lem:51} and Lemma~\ref{lem:testknot}(ii),(iii) give an $O(n_i)$ test: $S_i$ must be neither a pretzel nor a nutcracker shadow.

For $K=5_2$, Lemma~\ref{lem:52} and Lemma~\ref{lem:testknot}(i),(iv) give an $O(n_i)$ test: $S_i$ must be equivalent to neither a torus shadow nor $N(2,2)$.

Finally, for $K=6_2$, Lemmas~\ref{lem:62} and~\ref{lem:testknot} give an $O(n_i)$ test, using part~(iv) for $S(6_3)$ and $S(8_{18})$.

In every case the total running time is
$$
O(n)+O(n_1+\cdots+n_k)=O(n).
$$
This proves the theorem.
\end{proof}

\begin{proof}[Proof of Theorem~\ref{thm:links}]
We assume that we are given as input a Gauss paragraph of a link shadow $S$ with $n$ vertices. The goal is to show that if $L\in \{L2a1,L4a1,L5a1,L6n1\}$ then we can verify in $O(n)$ time whether $S$ resolves into $L$.

First we check whether $S$ has the correct number of components: two for $L2a1,L4a1$ and $L5a1$, and three for $L6n1$. If not, then $S$ cannot resolve into $L$, and so we stop.

Next we verify whether $S$ is connected. For this we note that $S$ is connected if and only if the graph whose vertices are the words of the Gauss paragraph of $S$, with two words being adjacent whenever they share a vertex label, is connected. This can also be decided in $O(n)$ time. If the answer is negative then $S$ only resolves into split links. Thus it cannot resolve into $L$, as none of $L2a1, L4a1, L5a1$, and $L6n1$ is split, and so we are done. Otherwise we move on.

Suppose first that $L$ is $L2a1$. In this case by Lemma~\ref{lem:L2a1} $S$ resolves into $L$, and so we are done.

Suppose now that $L$ is $L4a1$. In this case by Lemma~\ref{lem:L4a1} in order to see if $S$ resolves into $L$ it suffices to check whether the two components of $S$ share at least $4$ vertices. By Lemma~\ref{lem:test2link}(i) this can be verified in $O(n)$ time, and so we are done.

Suppose then that $L$ is $L5a1$. Since this is a prime link, we may invoke Lemma~\ref{lem:primeslinks} to compute in $O(n)$ time the Gauss paragraph of a connected prime $2$-component link shadow $S'$ with $m\le n$ vertices and with the property that $S$ resolves into $L$ if and only if $S'$ resolves into $L$. By Lemma~\ref{lem:L5a1}, $S'$ resolves into $L$ if and only if $S'$ is not equivalent to a torus shadow. By Lemma~\ref{lem:test2link}(ii) we can test this in $O(m)$ time, and hence in $O(n)$ time.

Suppose finally that $L$ is $L6n1$. In this case in view of Lemma~\ref{lem:L6n1} it suffices to verify whether $S$ is pairwise crossing. Lemma~\ref{lem:test3link} guarantees that this can be decided in $O(n)$ time, and so we are done.
\end{proof}

%%%%%%%%%%%%%%%%%%%%%%%%%%%%%%%%%%%%%%%%%%%%%%%%%%%%%%%%%%%%%%%%%%%%%%%%%%%
%%%%%%%%%%%%%%%%%%%%%%%%%%%%%%%%%%%%%%%%%%%%%%%%%%%%%%%%%%%%%%%%%%%%%%%%%%%
%%%%%%%%%%%%%%%%%%%%%%%%%%%%%%%%%%%%%%%%%%%%%%%%%%%%%%%%%%%%%%%%%%%%%%%%%%5
%%%%%%%%%%%%%%%%%%%%%%%%%%%%%%%%%%%%%%%%%%%%%%%%%%%%%%%%%%%%%%%%%%%%%%%%%%5
%%%%%%%%%%%%%%%%%%%%%%%%%%%%%%%%%%%%%%%%%%%%%%%%%%%%%%%%%%%%%%%%%%%%%%%%%%5
% ****************************************************************************************

% ****************************************************************************************
%%%%%%%%%%%%%%%%%%%%%%%%%%%%%%%%%%%%%%%%%%%%%%%%%%%%%%%%%%%%%%%%%%%%%%%%%%%
%%%%%%%%%%%%%%%%%%%%%%%%%%%%%%%%%%%%%%%%%%%%%%%%%%%%%%%%%%%%%%%%%%%%%%%%%%%
%%%%%%%%%%%%%%%%%%%%%%%%%%%%%%%%%%%%%%%%%%%%%%%%%%%%%%%%%%%%%%%%%%%%%%%%%%%

\section{Proofs of Lemmas~\ref{lem:primesknots} and~\ref{lem:primeslinks}}
\label{sec:primeproofs}

The proofs use the standard cactus representation of the $2$-edge cuts of a $2$-edge-connected graph.

\begin{lemma}[Prime decomposition of a shadow]\label{lem:shadowcactus}
Let $S$ be a connected shadow with $n\geq2$ vertices. Given its Gauss code, or its Gauss paragraph in the link case, one can compute in $O(n)$ time connected shadows $P_1,\ldots,P_r$ such that:
\begin{enumerate}
\item[(i)] $P_1,\ldots,P_r$ are obtained from $S$ by recursively splitting along $2$-edge cuts and joining the two severed ends on each side; 
\item[(ii)] every $P_i$ with at least two vertices has no $2$-edge cut; 
\item[(iii)]
$$
|P_1|+\cdots+|P_r|=n;
$$
\item[(iv)] for each $i=1,\ldots,r$ the Gauss code or paragraph of $P_i$ is obtained by restricting the input component words to $V(P_i)$ and (for Gauss paragraphs) discarding the empty words.
\end{enumerate}
\end{lemma}

For the proof we recall that a cactus is a connected graph in which every edge belongs to exactly one cycle.  The $2$-edge cuts of a $2$-edge-connected graph admit a cactus representation whose vertices correspond to its $3$-edge-connected components.  For this representation and its linear-time computation, see~\cite[Sections~10--11]{mehlhorn} and~\cite{tsin}.

\begin{proof}
Let $G$ be the underlying $4$-regular multigraph of $S$, and let $H$ be obtained from $G$ by deleting its loop edges (and keeping their incident vertices). Since $G$ has $2n$ edges, both $G$ and $H$ can be constructed from the input in $O(n)$ time.

Since $G$ is $4$-regular and connected, deleting its loop edges leaves $H$ connected and decreases every vertex degree by an even number. Thus $H$ is Eulerian, and hence has no cut-edges. Its $3$-edge-connected components $X_1,\ldots,X_r$, together with the cactus $C$ representing all its $2$-edge cuts can be computed in $O(n)$ time \cite{tsin,mehlhorn}.

Having computed $X_1,\ldots,X_r$ and the cactus $C$ we now move on to obtaining $P_1,\ldots,P_r$ and proving Properties (i)--(iv). As usual, we let $V(X_i)$ denote the vertex set of $X_i$, for $i=1,\ldots,r$.

For each $i=1,\ldots,r$ let $x_i$ be the vertex of $C$ corresponding to $X_i$.  Let $P_i$ be obtained from the subgraph of $G$ induced by $V(X_i)$ as follows. For every cycle $Q$ of $C$ containing $x_i$, the two edges of $Q$ incident with $x_i$ correspond to two edges of $H$ leaving $X_i$; add a {\em repair} edge joining their $X_i$-ends.  Each repair edge is precisely the edge created on the $X_i$-side by splitting along the corresponding $2$-edge cut.  Therefore recursively splitting along the cuts represented by the cactus produces precisely the connected shadows $P_1,\ldots,P_r$, as claimed in (i).

As with $X_i$, we use $V(P_i)$ to denote the vertex set of $P_i$. 

In order to prove (ii) we show the following for each $i=1,\ldots,r$: ($*$) {\em if $u,v$ are any two distinct vertices of $P_i$, then $u$ and $v$ are joined by three edge-disjoint paths in $P_i$.}

To prove ($*$) first we note that since $V(P_i)=V(X_i)$ then both $u$ and $v$ are in $X_i$. Now since $X_i$ is a $3$-edge-connected component of $H$ then there are three edge-disjoint $u$--$v$ paths in $H$.  In each of these paths we replace every maximal subpath whose internal vertices lie outside $X_i$ by the repair edge corresponding to the unique cycle of $C$ that contains the two cactus edges corresponding to the first and last edges of that subpath. This gives three $u$--$v$ paths contained in $P_i$. Moreover, these are edge-disjoint: indeed, if two of them used the same repair edge then the corresponding two paths in $H$ would share one of the two edges leaving $X_i$ that define that repair edge, which is impossible.  Thus ($*$) holds, and so (ii) follows.

To prove (iii) we note that the sets $V(X_1),\ldots,V(X_r)$ partition $V(H)=V(G)=V(S)$, and the repair operation creates no vertices. Since $V(P_i)=V(X_i)$ for $i=1,\ldots,r$, it follows that $$|P_1|+\cdots+|P_r|=n,$$ as claimed in (iii).

In order to prove (iv) let us consider one step in the recursive splitting process of (i), in which a shadow is split along a $2$-edge cut. The two cut edges belong to the same straight-ahead component, since each closed straight-ahead component uses an even number of edges of any cut. The key observation is that on each side the repair edge simply shortcuts the portion of that component lying on the other side, and so the component words on each side are obtained by restricting the previous component words to the vertices on that side. Moreover, the local configuration at every retained vertex does not change, and so in particular its sign is unchanged. Therefore iterating this reasoning over all the splits (iv) follows.

In order to see that indeed the whole construction takes $O(n)$ time we first note that the graph $H$, its $3$-edge-connected components, and the cactus $C$ are computed in $O(n)$ time. Since the cactus has $O(n)$ size it follows that all the repair edges can also be constructed in $O(n)$ time. Moreover, once the component containing each vertex is known, all the Gauss codes or paragraphs of all the shadows $P_i$ can be obtained in one scan of the component words of the original input. Thus the total running time is $O(n)$.
\end{proof}

\begin{proof}[Proof of Lemma~\ref{lem:primesknots}]
If $n\leq1$, return the empty collection. Assume $n\geq2$ and apply Lemma~\ref{lem:shadowcactus} (note that $S$ is connected, being the image of a single closed straight-ahead walk), obtaining $P_1,\ldots,P_r$. Since the input has one component word, Lemma~\ref{lem:shadowcactus}(iv) shows that every $P_i$ is a knot shadow.

Discard the one-vertex pieces and call the remaining shadows $S_1,\ldots,S_k$. By Lemma~\ref{lem:shadowcactus}(ii), they are prime, and
$$|S_1|+\cdots+|S_k|\leq n.$$
By Lemma~\ref{lem:shadowcactus}(i), $S$ is an iterated connected sum of the $P_i$. The discarded pieces resolve only into the unknot, so Theorem~\ref{thm:motiv} gives, for every nontrivial prime knot $K$,
$$ S\text{ resolves into }K \quad\Longleftrightarrow\quad S_i\text{ resolves into }K\text{ for some }i.$$
Finally, Lemma~\ref{lem:shadowcactus}(iv) gives all the required Gauss codes in total $O(n)$ time.
\end{proof}

\begin{proof}[Proof of Lemma~\ref{lem:primeslinks}]
Since $S$ is a connected $2$-component link shadow, it has at least two vertices. Apply Lemma~\ref{lem:shadowcactus}, obtaining $P_1,\ldots,P_r$.

Consider a split of a $2$-component link shadow along a $2$-edge cut. Each straight-ahead component uses an even number of cut edges, so one component uses both and the other uses neither. Thus one resulting piece is again a $2$-component link shadow and the other is a knot shadow. A split of a knot shadow produces two knot shadows. It follows inductively that exactly one of the $P_i$ is a $2$-component link shadow; call it $S'$.

The shadow $S'$ has at least two vertices, and by Lemma~\ref{lem:shadowcactus}(ii) it has no $2$-edge cut. Hence it is prime. By Lemma~\ref{lem:shadowcactus}(iii),
$$|S'|\leq n, $$
and Lemma~\ref{lem:shadowcactus}(iv) gives its Gauss paragraph in $O(n)$ time.

By Lemma~\ref{lem:shadowcactus}(i), reversing the splits reconstructs $S$ from $S'$ by inserting knot-shadow summands into its two components. These summands may all be resolved into the unknot, while in any resolution of $S$ that is a prime link they must all resolve into the unknot, since a connected sum of a link with a nontrivial knot is not prime. Therefore, for every prime link $L$,
$$S\text{ resolves into }L \quad\Longleftrightarrow\quad S'\text{ resolves into }L. $$
This proves the lemma.
\end{proof}

%%%%%%%%%%%%%%%%%%%%%%%%%%%%%%%%%%%%%%%%%%%%%%%%%%%%%%%%%%%%%%%%%%%%%%%%%%%
%%%%%%%%%%%%%%%%%%%%%%%%%%%%%%%%%%%%%%%%%%%%%%%%%%%%%%%%%%%%%%%%%%%%%%%%%%%
%%%%%%%%%%%%%%%%%%%%%%%%%%%%%%%%%%%%%%%%%%%%%%%%%%%%%%%%%%%%%%%%%%%%%%%%%%%

\section{Normal forms for Gauss codes and paragraphs of the relevant shadows}\label{sec:normalforms}

Throughout this section we will use for words the same notation introduced in Section~\ref{sec:gausscodes} for Gauss codes. That is, if
$$
A=a_1^{\varepsilon_1}\cdots a_k^{\varepsilon_k}
$$
is a word, then $\Neg{A}$ is obtained by changing every sign, and $\rev{A}$ is obtained by reversing the order of the letters. Thus
$$
\Neg{A}=a_1^{-\varepsilon_1}\cdots a_k^{-\varepsilon_k},
\qquad
\rev{A}=a_k^{\varepsilon_k}\cdots a_1^{\varepsilon_1}.
$$

(The following use of the word {\em alternating} is unrelated to its usual meaning for link diagrams.) We say that $A$ is {\em alternating} if $\varepsilon_i=-\varepsilon_{i+1}$ for every $i=1,\ldots,k-1$. If in addition $\varepsilon_{k}=-\varepsilon_1$ then $A$ is {\em cyclically alternating}.

\subsection{Gauss codes of torus shadows}\label{sub:torusgauss}

We start by identifying a criterion to verify whether a given knot shadow is equivalent to a torus shadow.

\begin{observation}\label{obs:torus1}
Let $S$ be a knot shadow with $p$ vertices, where $p$ is odd, and let $\Gamma$ be a Gauss code of $S$. Then $S$ is equivalent to $T(2,p)$ if and only if $\Gamma$ is congruent to
$$
A\,\Neg{A},
$$
\noindent for some alternating word $A$ of length $p$ whose vertex labels are pairwise distinct.
\end{observation}

\begin{proof}
With the vertex labelling and traversal illustrated in Figure~\ref{fig:torusgauss1}, we obtain for $T(2,p)$ the Gauss code
$$
A\Neg{A},
\qquad
A:=a_1^+a_2^-a_3^+\cdots a_p^+.
$$

Every alternating word of length $p$ with pairwise distinct vertex labels is obtained from $A$ by a relabelling and, possibly, negation. The result now follows from Proposition~\ref{pro:carter}.
\end{proof}

% **************************************************************
\begin{figure}[htbp] 
\def\ta#1{{\Scale[3.0]{#1}}} 
\def\tb#1{{\Scale[3.2]{#1}}} 
\def\tc#1{{\Scale[3.5]{#1}}} 
\def\somea{{\Scale[4.8]{\text{\rm (a)}}}} 
\def\somea{{\Scale[4.8]{\text{\rm (a)}}}} 
\def\someb{{\Scale[4.8]{\text{\rm (b)}}}} 
\def\somec{{\Scale[4.8]{\text{\rm (c)}}}} 
\def\somed{{\Scale[4.8]{\text{\rm (d)}}}} 
\def\somee{{\Scale[4.8]{\text{\rm (e)}}}} 
\def\somef{{\Scale[4.8]{\text{\rm (f)}}}} 
\def\tsa{\scalebox{8}{\small\rmfamily (Knot shadow)}}
\def\ltsa{\scalebox{8}{\small\rmfamily ($2$-component link shadow)}}
\centering 
\scalebox{0.15}{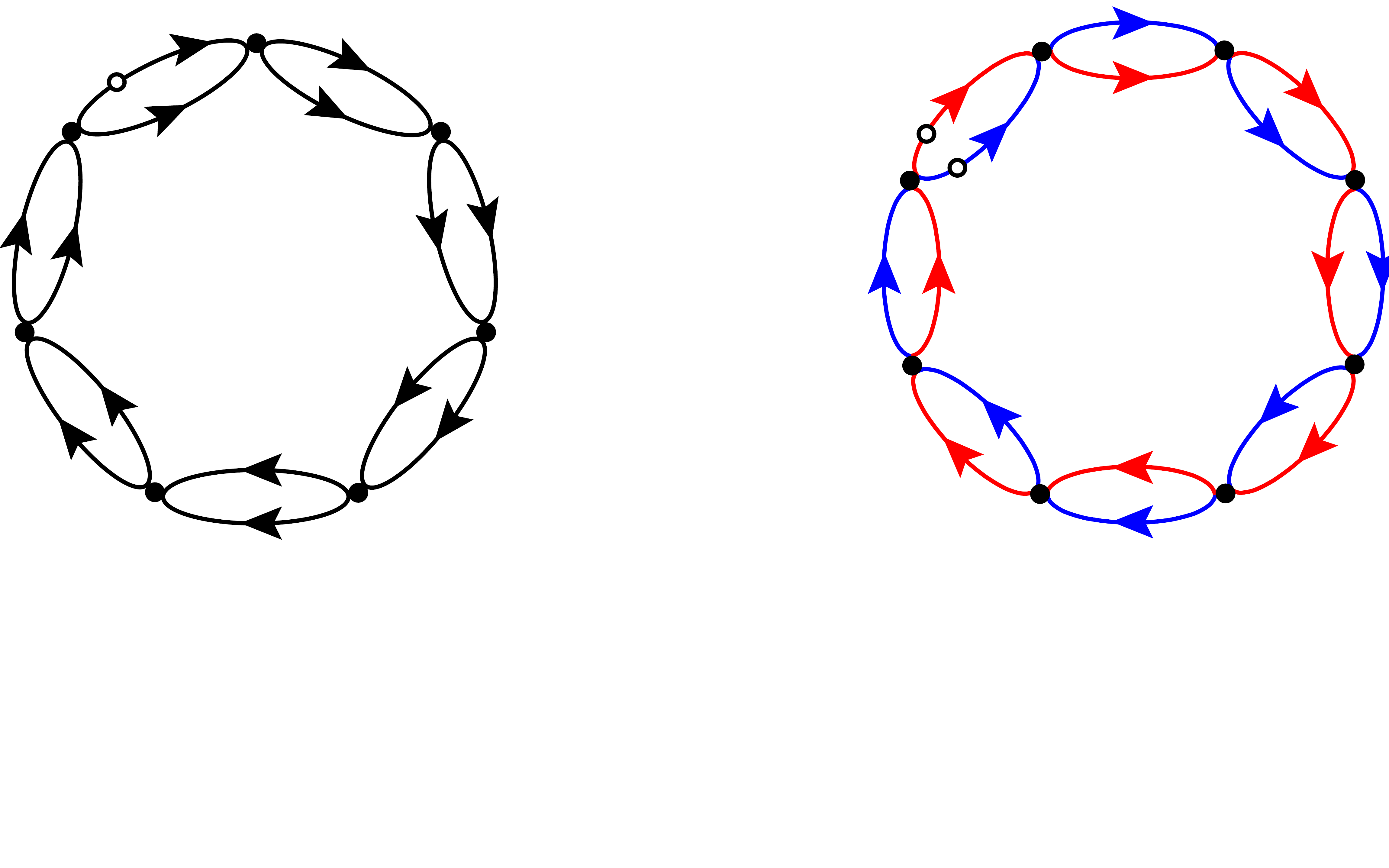}
\caption{Illustration of the proofs of Observations~\ref{obs:torus1} and~\ref{obs:torus2}.}
\label{fig:torusgauss1} 
\end{figure} 
% **************************************************************

\begin{corollary}[Identifying torus knot shadows]\label{cor:torus1}
Let $S$ be a knot shadow with $p$ vertices, where $p$ is odd. Let $\Gamma$ be a Gauss code of $S$, and let $B$ be the subword of $\Gamma$ that consists of its first $p$ symbols. Then $S$ is equivalent to $T(2,p)$  if and only if $B$ is alternating and $\Gamma=B\, \Neg{B}$.
\end{corollary}

\begin{proof}
Suppose first that $S$ is equivalent to $T(2,p)$. By Observation~\ref{obs:torus1}, $\Gamma$ is then congruent to $A\,\Neg{A}$, for some alternating word $A$ of length $p$.

Since $p$ is odd, (a) the word $A\,\Neg{A}$ is cyclically alternating. Moreover, (b) its $(i+p)$-th symbol is the negative of its $i$-th symbol, for every $i$ (indices read modulo $2p$). Clearly (a) and (b) are preserved by cyclic translation, reversal, negation, and relabelling, and so they are also satisfied in $\Gamma$. Now (a) implies that the first $p$ symbols of $\Gamma$ form an alternating word $B$, and (b) that $\Gamma=B\,\Neg{B}$.

Conversely, suppose that $B$ is alternating and $\Gamma=B\,\Neg{B}$. Since every label occurring in $B$ occurs again in $\Neg{B}$ and each label occurs exactly twice in $\Gamma$ it follows that the labels of $B$ are pairwise distinct. Therefore by Observation~\ref{obs:torus1} $S$ is equivalent to $T(2,p)$.
\end{proof}

% ****************************************************************************************
\subsection{Gauss codes of pretzel shadows}\label{sub:pre1}

We now identify whether an input shadow is a pretzel shadow.

\begin{observation}[Pretzel shadows]\label{obs:pretzel1}
Let $S$ be a knot shadow, and let $\Gamma$ be a Gauss code of $S$. Then $S$ is equivalent to a pretzel shadow if and only if $\Gamma$ is congruent to
$$
A\,B\,C\,\Neg{\rev{A}}\,\Neg{\rev{B}}\,\Neg{\rev{C}},
$$
where each of $A,B$ and $C$ has odd length, $ABC$ is alternating, and the vertex labels of $ABC$ are pairwise distinct.
\end{observation}

\begin{proof}
With the vertex labelling and traversal illustrated in Figure~\ref{fig:pretzelgauss1}, we obtain for $S(p_1,p_2,p_3)$ the Gauss code
$$
A\,B\,C\,\Neg{\rev{A}}\,\Neg{\rev{B}}\,\Neg{\rev{C}},
$$
\noindent where
$$A := a_1^+ a_2^- \cdots a_{p_1}^+, \,\, B:= b_1^- b_2^+ \cdots b_{p_2}^-,  \,\, \text{\rm and }\, \, C:= c_1^+ c_2^- \cdots c_{p_3}^+.$$
Every choice of $A,B,C$ satisfying the conditions in the statement is obtained from this one by relabelling the vertices and, possibly, negating all signs. The observation now follows from Proposition~\ref{pro:carter}.
\end{proof}

% **************************************************************
\begin{figure}[htbp] 
\def\ta#1{{\Scale[5.0]{#1}}} 
\def\tb#1{{\Scale[3.6]{#1}}} 
\def\td#1{{\Scale[4.6]{#1}}} 
\def\Ap#1{{\Scale[4.6]{a_{{}_{#1}}}}} 
\def\Bp#1{{\Scale[4.6]{b_{{}_{#1}}}}} 
\def\Cp#1{{\Scale[4.6]{c_{{}_{#1}}}}} 
\def\somea{{\Scale[4.8]{\text{\rm (a)}}}} 
\def\someb{{\Scale[4.8]{\text{\rm (b)}}}} 
\def\somec{{\Scale[4.8]{\text{\rm (c)}}}} 
\def\somed{{\Scale[4.8]{\text{\rm (d)}}}} 
\def\somee{{\Scale[4.8]{\text{\rm (e)}}}} 
\def\somef{{\Scale[4.8]{\text{\rm (f)}}}} 
\def\psa{\scalebox{11}{\small\rmfamily Pretzel shadow}}
\centering 
\scalebox{0.1}{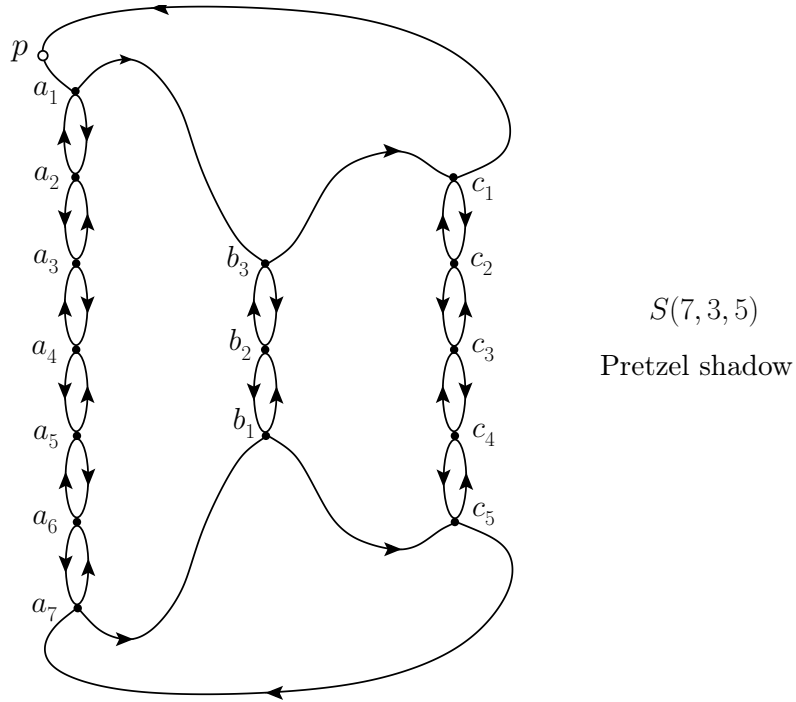}
\caption{Illustration of the proof of Observation~\ref{obs:pretzel1}.}
\label{fig:pretzelgauss1} 
\end{figure} 
% **************************************************************

\begin{corollary}[Identifying pretzel shadows]\label{cor:pretzel1}
Let $S$ be a knot shadow, and let $\Gamma$ be a Gauss code of $S$. Then $S$ is equivalent to a pretzel shadow if and only if there is a cyclic translation of $\Gamma$ of the form $A\,B\,C\,\Neg{\rev{A}}\,\Neg{\rev{B}}\,\Neg{\rev{C}}$, where $A,B$ and $C$ have odd length and $ABC$ is alternating.
\end{corollary}

\begin{proof}
By Observation~\ref{obs:pretzel1}, it suffices to note that negation, reversal, and relabelling preserve the stated form, up to cyclic translation. For reversal, one has
$$
\rev{\bigl(
ABC\Neg{\rev A}\Neg{\rev B}\Neg{\rev C}
\bigr)}
=
\Neg{C}\,\Neg{B}\,\Neg{A}\,\rev{C}\,\rev{B}\,\rev{A},
$$
\noindent which, since $A,B,C$ have odd length and $ABC$ is alternating, is again of the required form. Finally, the vertex labels of $ABC$ are automatically pairwise distinct, as a label occurring twice in $ABC$ would occur four times in $\Gamma$.
\end{proof}

% ****************************************************************************************
\subsection{Gauss codes of nutcracker shadows}\label{sub:nut1}

We finally lay out how to identify whether an input shadow is a nutcracker shadow.

\begin{observation}[Nutcracker shadows]\label{obs:nut1}
Let $S$ be a knot shadow, and let $\Gamma$ be a Gauss code of $S$. Then $S$ is equivalent to a nutcracker shadow if and only if $\Gamma$ is congruent to
$$
A\,B\,\Neg{\rev{A}}\,\Neg{\rev{B}},
$$
where $A,B$ are words of positive even length, each of $A$ and $B$ is alternating, the last entry of $A$ and the first entry of $B$ have the same sign, and the vertex labels of $AB$ are pairwise distinct.
\end{observation}

\begin{proof}
With the vertex labelling and traversal illustrated in Figure~\ref{fig:nutgauss1}, we obtain for $N(2s,2t)$ the Gauss code
$$
A\,B\,\Neg{\rev{A}}\,\Neg{\rev{B}},
$$
where
$$
A:=a_1^+a_2^-\cdots a_{2s}^-,
\qquad
B:=b_1^-b_2^+\cdots b_{2t}^+.
$$
Every choice of $A,B$ satisfying the conditions in the observation is obtained from this one by relabelling the vertices and possibly negating all signs. Thus the observation follows from Proposition~\ref{pro:carter}.
\end{proof}

% **************************************************************
\begin{figure}[htbp] 
\def\ta#1{{\Scale[5.0]{#1}}} 
\def\tb#1{{\Scale[3.6]{#1}}} 
\def\td#1{{\Scale[4.6]{#1}}} 
\def\Ap#1{{\Scale[4.6]{a_{{}_{#1}}}}} 
\def\Bp#1{{\Scale[4.6]{b_{{}_{#1}}}}} 
\def\somea{{\Scale[4.8]{\text{\rm (a)}}}} 
\def\someb{{\Scale[4.8]{\text{\rm (b)}}}} 
\def\somec{{\Scale[4.8]{\text{\rm (c)}}}} 
\def\somed{{\Scale[4.8]{\text{\rm (d)}}}} 
\def\somee{{\Scale[4.8]{\text{\rm (e)}}}} 
\def\somef{{\Scale[4.8]{\text{\rm (f)}}}} 
\def\psa{\scalebox{11}{\small\rmfamily Nutcracker shadow}}
\centering 
\scalebox{0.1}{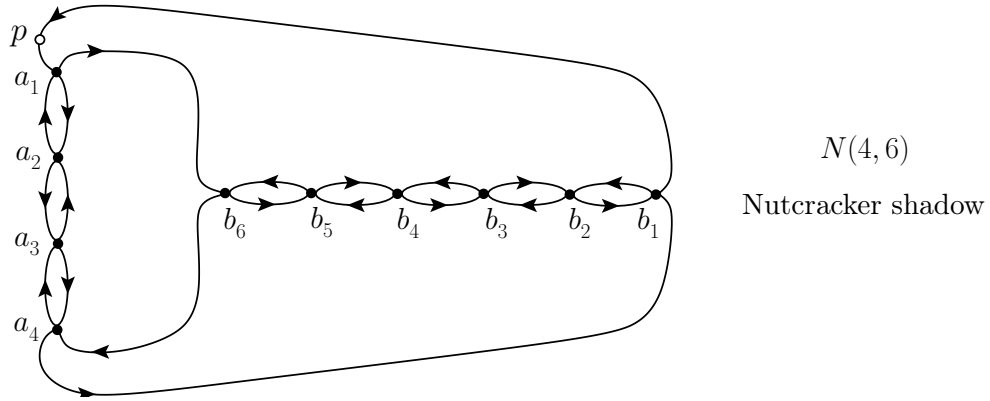}
\caption{Illustration of the proof of Observation~\ref{obs:nut1}.}
\label{fig:nutgauss1} 
\end{figure} 
% **************************************************************

\begin{corollary}[Identifying nutcracker shadows]\label{cor:nut1}
Let $S$ be a knot shadow, and let $\Gamma$ be a Gauss code of $S$. Then $S$ is equivalent to a nutcracker shadow if and only if there is a cyclic translation of $\Gamma$ of the form
$$
A\,B\,\Neg{\rev{A}}\,\Neg{\rev{B}},
$$
where $A$ and $B$ have positive even length, each of $A$ and $B$ is alternating, and the last entry of $A$ and the first entry of $B$ have the same sign.
\end{corollary}

\begin{proof}
By Observation~\ref{obs:nut1}, it suffices to note that negation, reversal, and relabelling preserve the stated form, up to cyclic translation. For reversal, one has
$$
\rev{\bigl(
AB\Neg{\rev A}\Neg{\rev B}
\bigr)}
=
\Neg B\,\Neg A\,\rev B\,\rev A,
$$
which, since $A,B$ have even length and satisfy the stated sign condition, is again of the required form. Finally, the vertex labels of $AB$ are automatically pairwise distinct, as a label occurring twice in $AB$ would occur four times in $\Gamma$.
\end{proof}

% ****************************************************************************************
\subsection{Gauss paragraphs  of two-component torus shadows}\label{sub:torus2}

\begin{observation}[Two-component torus shadows]\label{obs:torus2}
Let $S$ be a connected $2$-component link shadow with $p$ vertices, where $p$ is even, and let $\BG$ be a Gauss paragraph of $S$. Then $S$ is equivalent to $T(2,p)$ if and only if $\BG$ is congruent to
$$
\{A,\Neg{A}\},
$$
for some alternating word $A$ of length $p$ whose vertex labels are pairwise distinct.
\end{observation}

\begin{proof}
As illustrated for $T(2,8)$ in Figure~\ref{fig:torusgauss1}, suitable choices of basepoints and traversal directions give for $T(2,p)$ the Gauss paragraph
$$
\{A,\Neg{A}\},
\qquad
A:=a_1^+a_2^-a_3^+\cdots a_p^-.
$$
Every alternating word of length $p$ with pairwise distinct vertex labels is obtained from $A$ by a relabelling and, possibly, negation. The result follows from Proposition~\ref{pro:carterlinks}.
\end{proof}

\begin{corollary}[Identifying two-component torus shadows]
\label{cor:torus2}
Let $S$ be a $2$-component link shadow with $p$ vertices, where $p$ is even, and let $\{\Gamma_1,\Gamma_2\}$ be a Gauss paragraph of $S$. Then $S$ is equivalent to $T(2,p)$ if and only if $\Gamma_1$ is alternating and $\Gamma_2$ is a cyclic translation of either
$$
\Neg{\Gamma_1}
\qquad\text{or}\qquad
\Neg{\rev{\Gamma_1}}.
$$
\end{corollary}

\begin{proof}
The assumptions on $\Gamma_1$ and $\Gamma_2$ imply that $\Gamma_1$ and $\Gamma_2$ have the same vertex labels, and so every vertex is shared by the two components and $S$ is connected. The corollary then follows from Observation~\ref{obs:torus2} and the definition of congruence. Indeed, a single reversal reverses one component word, and since every vertex is shared this action changes the signs in both words, thus yielding the two alternatives above.
\end{proof}

% ****************************************************************************************

% ****************************************************************************************
% ****************************************************************************************
\section{Proofs of Lemmas~\ref{lem:testknot},~\ref{lem:test2link}, and~\ref{lem:test3link}}\label{sec:equivtests}

We first make one simple observation that will be useful for testing the pretzel and nutcracker normal forms of Section~\ref{sec:normalforms}. Let
$$
\Gamma=x_1x_2\cdots x_{2n},
$$
where indices are read modulo $2n$. For each $i$, let $m(i)$ be the position of the other occurrence of the vertex label of $x_i$. We call the gap after $x_i$ a {\em break} if
$$
m(i+1)\not\equiv m(i)-1\pmod {2n}.
$$
The values $m(i)$, and hence all the breaks, can clearly be computed in $O(n)$ time.

The pretzel and the nutcracker normal forms are both of the form $X_1\cdots X_k\,\Neg{\rev{X_1}}\cdots\Neg{\rev{X_k}}$. For a Gauss code of this form, if we write $q_j:=|X_j|$ and $Q_j:=q_1+\cdots+q_j$, so that $n=Q_k$, we have that
$$
m(Q_{j-1}+r)=n+Q_j+1-r \qquad \text{for } 1\le r\le q_j.
$$

Using this identity we easily obtain the following crucial observation.

\begin{observation}\label{obs:breaks}
Let $k\ge2$, let $X_1,\ldots,X_k$ be nonempty words, and suppose that
$$
\Gamma=X_1\cdots X_k\,\Neg{\rev{X_1}}\cdots\Neg{\rev{X_k}}
$$
is a Gauss code. Then the breaks of $\Gamma$ are exactly the $2k$ boundaries between the displayed blocks. In particular, a Gauss code in the pretzel normal form of Corollary~\ref{cor:pretzel1} has exactly six breaks, and a Gauss code in the nutcracker normal form of Corollary~\ref{cor:nut1} has exactly four breaks.
\end{observation}

\begin{proof}[Proof of Lemma~\ref{lem:testknot}(i)]
We first test the parity of $n$: if $n$ is even, then $S$ is not a torus shadow, and so we are done. Suppose then that $n$ is odd, and let $B$ be the first $n$ symbols of $\Gamma$. By Corollary~\ref{cor:torus1} it suffices to check that $B$ is alternating and that the last $n$ symbols of $\Gamma$ form exactly the word $\Neg{B}$. Since both conditions can be verified in $O(n)$ time, we are done.
\end{proof}

\begin{proof}[Proof of Lemma~\ref{lem:testknot}(ii)]
In order to use Observation~\ref{obs:breaks} we start by computing all the breaks of $\Gamma$. It is easy to see that this can be done in $O(n)$ time.

By Corollary~\ref{cor:pretzel1} and Observation~\ref{obs:breaks}, if $S$ is a pretzel shadow then $\Gamma$ has exactly six breaks,
namely the boundaries of
$$
A,\ B,\ C,\ \Neg{\rev A},\ \Neg{\rev B},\ \Neg{\rev C}.
$$
If there are not exactly six breaks then there is nothing else to do. Thus let us assume that there are indeed exactly six breaks.

There are only six possible cyclic translations to consider, one beginning immediately after each break. Fix one of these six cyclic translations, and let $D_1,\ldots,D_6$ be the six blocks determined by the breaks. We now check whether
$$
|D_1|,\ |D_2|,\ |D_3|
$$
are odd, whether $D_1D_2D_3$ is alternating, and whether
$$
D_4=\Neg{\rev{D_1}},\qquad
D_5=\Neg{\rev{D_2}},\qquad
D_6=\Neg{\rev{D_3}}.
$$

By Corollary~\ref{cor:pretzel1} $S$ is equivalent to a pretzel shadow if and only if one of these translations satisfies these conditions. Since each translation is checked in $O(n)$ time and there are only six of them, the total running time is $O(n)$, as required.
\end{proof}

\begin{proof}[Proof of Lemma~\ref{lem:testknot}(iii)]
The proof is very similar to the proof of (ii). In this case we use Observation~\ref{obs:breaks} and Corollary~\ref{cor:nut1}, and so we need to check whether $\Gamma$ has exactly four breaks. Clearly this can be done in $O(n)$ time, and if the answer is negative then there is nothing else to do.

If $\Gamma$ has exactly four breaks then there are four cyclic translations to consider. We fix one of them and let $D_1,D_2,D_3,D_4$ be the four blocks determined by the breaks. In this case we need to check whether $D_1$ and $D_2$ have positive even length, whether they are both alternating, whether the last entry of $D_1$ and the first entry of $D_2$ have the same sign, and finally whether
$$
D_3=\Neg{\rev{D_1}},
\qquad
D_4=\Neg{\rev{D_2}}.
$$
By Corollary~\ref{cor:nut1} $S$ is equivalent to a nutcracker shadow if and only if one of these four translations satisfies all these conditions. Since each translation is checked in $O(n)$ time and there are four translations to check, the total running time is then $O(n)$.
\end{proof}

\begin{proof}[Proof of Lemma~\ref{lem:testknot}(iv)]
Let $\Gamma_0$ be any Gauss code of $S_0$, and let $n_0:=|S_0|$. By Proposition~\ref{pro:carter} it suffices to check whether the input code $\Gamma$ is congruent to $\Gamma_0$. Since $n_0$ is constant this can clearly be done in constant time, and so we are done.
\end{proof}

\begin{proof}[Proof of Lemma~\ref{lem:test2link}(i)]
We scan the two component words and for each vertex $v$ record which of the two words contains $v$. Note that $v$ is shared by the two components if and only if it occurs once in each word, and so in one scan we can count the total number of shared vertices. The test succeeds if and only if this number is at least four. Finally we note that it is easy to see that the whole process takes $O(n)$ time.
\end{proof}

\begin{proof}[Proof of Lemma~\ref{lem:test2link}(ii)]
Let $\{\Gamma_1,\Gamma_2\}$ be the given input Gauss paragraph. If $n$ is odd then $S$ is not a link torus shadow, and we are done. We may then assume that $n$ is even.

In one scan we verify whether each vertex occurs once in each component word. If this is not the case then we reject the input, and so we are done. Otherwise we continue, with the information that each of $\Gamma_1$ and $\Gamma_2$ has length $n$ and contains each vertex exactly once.

Next we verify whether $\Gamma_1$ is alternating. Since its vertex labels are necessarily pairwise distinct, the first letter of $\Gamma_2$ determines the only possible cyclic alignment with $\Neg{\Gamma_1}$, and also the only possible cyclic alignment with $\Neg{\rev{\Gamma_1}}$. We check these two alignments in one scan each. By Corollary~\ref{cor:torus2} $S$ is a torus shadow if and only if one of them succeeds. This last test also takes $O(n)$ time, and so the total running time is $O(n)$.
\end{proof}

\begin{proof}[Proof of Lemma~\ref{lem:test3link}]
We scan the three component words and for each vertex $v$ record the set of components in which $v$ occurs. Since every vertex occurs exactly twice in the whole Gauss paragraph, a vertex shared by two distinct components occurs once in each of them. During the scan we record which of the three pairs of components share a vertex. The required condition holds if and only if all three pairs satisfy this property. The whole procedure clearly takes $O(n)$ time, and so we are done.
\end{proof}

% ****************************************************************************************
\section{An open question}\label{sec:concluding}

The linear-time algorithms claimed in Theorems~\ref{thm:knots} and~\ref{thm:links} make essential use of the characterizations of the shadows that resolve into a knot in $\{3_1,4_1,5_1,5_2,6_2\}$ and of the shadows that resolve into a link in $\{L2a1,L4a1,L5a1,L6n1\}$. A glance at these characterizations suggests an obvious question:

\vglue 0.4 cm
\noindent{\bf Question. }{\em Let $L$ be a link. Is there a good characterization of the shadows that resolve into $L$?}
\vglue 0.4 cm

An efficiently testable characterization of this kind would imply a polynomial-time algorithm to decide, for each fixed link $L$, whether a given input shadow resolves into $L$. We intentionally write ``polynomial-time'' instead of ``linear-time'' here, as it is conceivable that the running time in Theorems~\ref{thm:knots} and~\ref{thm:links} cannot be matched for arbitrary links, even with a positive answer to this question.

% ****************************************************************************************

\bibliographystyle{abbrv} 
\bibliography{refs} 
\end{document}